\documentclass[11pt,a4paper]{article}
\usepackage[utf8]{inputenc}
\usepackage[margin=2cm]{geometry}
\usepackage{amsmath,amssymb,amsthm}
\usepackage{url}
\newtheorem{theorem}{Theorem}[section]
\newtheorem{proposition}[theorem]{Proposition}
\newtheorem{corollary}[theorem]{Corollary}
\newtheorem{lemma}[theorem]{Lemma}

\newtheorem{definition}[theorem]{Definition}
\newtheorem{remark}[theorem]{Remark}

\newcommand{\mc}{\mathcal}
\newcommand{\mf}{\mathfrak}

\newcommand{\aff}{\operatorname{aff}}

\newcommand{\hh}{\widehat}

\begin{document}
\title{Completely bounded homomorphisms of the Fourier algebra revisited}
\author{Matthew Daws}
\maketitle

\begin{abstract}
Let $A(G)$ and $B(H)$ be the Fourier and Fourier-Stieltjes algebras of locally compact groups $G$ and $H$, respectively.  Ilie and Spronk have shown that continuous piecewise affine maps $\alpha: Y \subseteq H\rightarrow G$ induce completely bounded homomorphisms $\Phi:A(G)\rightarrow B(H)$, and that when $G$ is amenable, every completely bounded homomorphism arises in this way.  This generalised work of Cohen in the abelian setting.  We believe that there is a gap in a key lemma of the existing argument, which we do not see how to repair.  We present here a different strategy to show the result, which instead of using topological arguments, is more combinatorial and makes use of measure theoretic ideas, following more closely the original ideas of Cohen.
\end{abstract}

\section{Introduction}

Cohen in \cite{cohen} classified all bounded homomorphisms from the group algebra $L^1(G)$ to the measure algebra $M(H)$, for locally compact abelian groups $G,H$; this was later expounded with different proofs by Rudin in \cite{rudinbook}.  The characterisation given was in terms of Pontryagin duals, and so in modern language is more naturally stated as studying homomorphisms between the Fourier algebra $A(\hh G)$ and the Fourier--Stieltjes algebra $B(\hh H)$; these algebras being introduced by Eymard \cite{eymard} for arbitrary locally compact groups.  It is now widely recognised that it is natural to work in the category of operator spaces and completely bounded maps when studying Fourier algebras of non-abelian groups.  In \cite{ilie} Ilie provided a generalisation of Cohen's result for discrete groups, characterising completely bounded homomorphisms $A(G)\rightarrow B(H)$ in terms of coset rings of $H$, and piecewise affine maps.  In \cite{is} Ilie and Spronk extended this result to all locally compact groups, making use of open coset rings.  We also mention \cite{pham} which shows similar results for merely contractive (not completely bounded) homomorphisms.

We do not fully follow the proof given in \cite{is}, nor than of Rudin in \cite{rudinbook}.  The proof of a key lemma in \cite{is} appears to have a gap, and we have been unable to see how to repair this.  In this paper, we return to Cohen's original proof for inspiration, and provide a new proof of the main result of \cite{is}, using a more complicated combinatorial argument, and using measure theory ideas.  Cohen's proof is, in places, firmly a proof about abelian groups, which necessitates new ideas in the non-abelian setting.

The papers \cite{ilie, is} have been widely cited, used and generalised in the 15 years since they were published, and we feel it is important that this result has a solid, careful proof attached to it.  Let us now provide some further background, which will allow us to be precise about the perceived problems in \cite{is, rudinbook}.  We will then give an overview of our alternative strategy.  The rest of the paper is concerned with the precise details of executing this new strategy.

Let $X$ be a set.  For us, a \emph{ring of subsets} of $X$, say, $\mc S$, is a (non-empty) collection of subsets of $X$ such that for $A,B\in\mc S$, also $A\cap B, A \cup B\in\mc S$, and if $A\in\mc S$ also $X\setminus A\in\mc S$.  Then $\emptyset$, and so also $X$, are in $\mc S$, and notice that $\mc S$ is also closed under taking symmetric differences.

Let $G$ be a group, and let $H\leq G$ be a subgroup.  A coset of $H$ is a left coset, $s_0H$, for some $s_0\in G$.  We remark that right cosets $Hs_0 = s_0(s_0^{-1}Hs_0)$ are left cosets of the (possibly different) subgroup $s_0^{-1}Hs_0$.  If $C = s_0H$ is a coset, then $C^{-1} C = H$ and $C C^{-1} C = C$.  With $G_1$ another group, a map $\alpha:H\rightarrow G_1$ is \emph{affine} if $\alpha(rs^{-1}t) = \alpha(r) \alpha(s)^{-1} \alpha(t)$ for $r,s,t\in H$.  This is equivalent to $H\rightarrow G_1; s \mapsto \alpha(s_0)^{-1} \alpha(s_0s)$ being a group homomorphism.  Given a subset $A\subseteq G$, let $\aff(A)$ be the smallest coset containing $A$.

The \emph{coset ring} of $G$, denoted $\Omega(G)$, is the smallest ring of subsets of $G$ containing all cosets of all subgroups of $G$.  Given $Y\subseteq G$ a map $\alpha:Y\rightarrow G_1$ is \emph{piecewise affine} when $Y$ is the finite disjoint union of sets $(Y_i)_{i=1}^n$ in $\Omega(G)$, and for each $i$ there is an affine map $\alpha_i:\aff(Y_i)\rightarrow G_1$ with $\alpha|_{Y_i} = \alpha_i|_{Y_i}$.  See \cite[Sections~2,~4]{ilie} and \cite[Section~1.2]{is} for combinatorial details about cosets and $\Omega(G)$.

Now let $G$ be a locally compact group, and let $\Omega_o(G)$ be the ring of sets generated by open cosets of $G$.  As an open subgroup is also closed, the same applies to open cosets, and so every member of $\Omega_o(G)$ is clopen.  The key lemma in \cite{is} is Lemma~1.3(ii), which states that with $G_1$ another locally compact group, if $\alpha:Y\rightarrow G_1$ is piecewise affine, and $\alpha$ is continuous, and $Y$ is open, then $\alpha$ has a continuous extension to $\overline\alpha:\overline Y\rightarrow G_1$.  Furthermore, $\overline Y$ is open, and in the decomposition $\overline Y = \bigsqcup_i Y_i$ we may assume that each $Y_i\in\Omega_o(G)$, and for each $i$ there is a open coset $C_i$ containing $Y_i$, and a continuous affine map $\alpha_i:C_i\rightarrow G_1$ which agrees with $\overline\alpha$ on $Y_i$.  The use of this lemma is that it allows us to combine the algebraic property that $\alpha$ is piecewise affine with the topological property that $\alpha$ is continuous, and conclude that $\alpha$ is the ``union'' of continuous affine maps.

In the proof of \cite[Lemma~1.3(ii)]{is} we have a coset $K$ and subcosets $N_1,\cdots,N_k$ of infinite index, and it is claimed that if $Y = K \setminus \bigcup_j N_j$ with $\overline Y$ having non-empty interior, then $\overline Y = \overline K \setminus \bigcup_{j\in J} N_j$ where $J$ is the collection of indices with $N_j$ having non-empty interior.  This is not true.  A counter-example from \cite{mopost} exhibits a compact abelian group $G_0$ and an index 2 subgroup $H_0$ so that both $H_0$ and $H_1=G_0\setminus H_0$ have empty interior, and yet of course $G_0 = H_0 \sqcup H_1$.   If we set $G = K = G_0 \times\mathbb Z$ and $N_i = H_i \times \{0\}$ then each $N_i$ has infinite index in $K$, each has empty interior, $Y = K\setminus (N_0\cup N_1) = G_0 \times (\mathbb Z\setminus\{0\})$ is already clopen, as is $K$, and yet $\overline Y \not= \overline K$.  The same issue can be seen in the middle of \cite[Section~4.5.2]{rudinbook}.

This counter-example does not mention (piecewise) affine maps, but we could simply let $\alpha$ be the identity.  The moral seems to be that we chose a ``silly'' way to write $\alpha$ as a piecewise affine map.  However, given an arbitrary piecewise affine map $\alpha$ which we happen to know is continuous, we need some argument to show that we can exhibit that $\alpha$ is piecewise affine in a ``sensible'' way.  Cohen's original argument uses knowledge about the graph of $\alpha$, and then a delicate combinatorial argument to show that we can exhibit the graph using sets built from the graph itself (compare Theorem~\ref{thm:one} below).  This result can then be used to show that we can exhibit that $\alpha$ is piecewise affine using at least measurable sets (the subgroup $H_0$ in the example above is not measurable) which together with a measure-theoretic argument then yields an analogue of \cite[Lemma~1.3(ii)]{is}.  We will use exactly the same general approach, but adapted to possible non-abelian groups.

Some of our argument closely follows Cohen's paper \cite{cohen}.  We must say that we find 
many of Cohen's arguments rather hard to follow.  In particular, our key technical result, Proposition~\ref{prop:3}, is similar to the lemma on \cite[Page~223--224]{cohen}, the proof of which we do not understand.  From our limited understanding, it seems clear, however, that this lemma of Cohen requires at least that every subgroup involved be normal (which is automatic, if the groups are abelian!)  Given that we need to check that all results hold for non-abelian groups, and that our central argument is entirely new, we have decided to give full details for all our results.  We indicate in a number of places where we follow Cohen quite closely.

\subsection{Acknowledgments}

We thank Yemon Choi for useful remarks on a preprint of this paper, and the referee for helpful comments
which have improved the clarity of the paper.
The author is partially supported by EPSRC grant EP/T030992/1.

\section{Initial setup of the problem}\label{sec:2}

We fix locally compact groups $G,H$ and a completely bounded homomorphism $\Phi:A(G)\rightarrow B(H)$.  Following the proof of \cite[Theorem~3.7]{is}, there is a continuous map $\alpha:H\rightarrow G_\infty$ with $\Phi(u)(s) = u(\alpha(s))$ for each $s\in G_\infty$.  Here $G_\infty$ is either the one-point compactification of $G$, if $G$ is not compact, or the disjoint union $G \sqcup \{\infty\}$ if $G$ is compact.  We extend each $u\in A(G)$ to a (continuous) function on $G_\infty$ by setting $u(\infty)=0$.  Then $Y = \alpha^{-1}(G)$ is open in $H$.  Under the additional hypothesis that $G$ is amenable, $\alpha:Y\rightarrow G$ is piecewise affine, if we regard $H$ and $G$ as just groups, with no topology.  In what follows, we shall not use amenability again.

From now on, $G,H$ will be arbitrary groups, with additional hypotheses stated as needed.  Given $Y\subseteq H$ and $\alpha:Y\rightarrow G$ we shall in the sequel write $\alpha: Y\subseteq H \rightarrow G$.  The $\emph{graph}$ of $\alpha$ is
\[ \mc G(\alpha) = \{ (s,\alpha(s)) : s\in Y \} \subseteq H\times G. \]
An extremely useful result is the following.

\begin{lemma}[{\cite[Lemma~1.2]{is}}]\label{lem:2}
Let $\alpha:Y\subseteq H\rightarrow G$ be a map.  Then $\mc G(\alpha) \in \Omega(H\times G)$ if and only if $\alpha$ is piecewise affine.
\end{lemma}

In the next section, we shall prove our main result, Theorem~\ref{thm:one}.  In the following section, we apply this to show that $\alpha$ can be exhibited as a piecewise affine map with the component sets involved at least being Borel, Proposition~\ref{prop:2}.  In the $\sigma$-finite case, a measure theoretic argument then yields what we want, and can be bootstrapped into a proof in the general case, Theorem~\ref{thm:2}.

\section{Combinatorial Lemma}

We begin by making a non-standard, but useful, definition.  (This definition is sometimes termed the ``measure theory ring of sets'', but we shall stick to our \emph{ad hoc} definition for clarity.)

\begin{definition}\label{defn:1}
Let $X$ be a set, and $\mc S$ a collection of subsets of $X$.  We say that $\mc S$ is a \emph{relative ring of subsets} when $\mc S$ is closed under finite unions, intersections, and relative complements, in the sense that if $A,B\in\mc S$ then $A\setminus B\in\mc S$.  This means that $\emptyset\in\mc S$ but perhaps $X$ is not in $\mc S$.

Let $G$ be a group, and $\alpha$ a collection of subsets of $G$.  Let $\mc R(\alpha)$ be the relative ring of subsets generated by $\alpha$ and all left translations of elements of $\alpha$.  Let $\mc R_2(\alpha)$ be the relative ring of subsets generated by $\alpha$ and all two sided translates of $\alpha$, that is, sets of the form $sAt$ where $A\in\alpha, s,t\in G$.
\end{definition}

This section is devoted to proving the following result.  As to why we work with two sided translates and not just left translates, see Remark~\ref{rem:1} below.

\begin{theorem}\label{thm:one}
Let $G$ be a group and let $Y\in\Omega(G)$.  There are subgroups $H_1,H_2,\cdots,H_n$ in $\mc R_2(\{Y\})$ such that $Y\in\mc R(\{H_1,\cdots,H_n\})$.
\end{theorem}

We first of all collect some combinatorial lemmas.  These results are similar in spirit to arguments found in \cite[Section~3]{cohen}.  In the following, the empty intersection is by definition equal to $X$, and the empty union equal to $\emptyset$.

\begin{lemma}\label{lem:3}
Let $X$ be a set, let $\mf A$ be a collection of subsets of $X$, and let $B$ be in the relative ring generated by $\mf A$.  Then $B$ is the finite (disjoint, if we wish) union of sets of the form
\begin{equation}
\bigcap_{i=1}^n A_i \cap \bigcap_{j=1}^m (X\setminus B_j)
= \Big(\bigcap_{i=1}^n A_i\Big) \setminus \Big( \bigcup_{j=1}^m B_j \Big)
 \label{eq:1}
\end{equation}
for some $n\geq 1, m\geq 0$, and some $A_i,B_j\in\mf A$.
\end{lemma}
\begin{proof}
Let $\mf B$ be the relative ring generated by $\mf A$, so $B\in\mf B$.  Let $\mf B'$ be the collection of finite unions of sets of the form (\ref{eq:1}), so that $\mf A \subseteq \mf B' \subseteq \mf B$.  We shall prove that $\mf B'$ is a relative ring of sets, so that $\mf B' = \mf B$, as claimed.

By definition, $\mf B'$ is closed under unions.  If $(P_k)_{k=1}^t$ and $(Q_l)_{l=1}^s$ are of the form (\ref{eq:1}), set $P=\bigcup_k P_k$ and $Q=\bigcup_l Q_l$, then
\[ P \cap Q = (P_1\cup\cdots\cup P_t) \cap (Q_1\cup\cdots\cup Q_s) = \bigcup_{k,l} P_k \cap Q_l \]
and clearly $P_k\cap Q_l$ is of the form (\ref{eq:1}).  So $\mf B'$ is closed under intersections.  Furthermore,
\[ P \setminus Q = P \cap (X\setminus Q) = P\cap (X\setminus Q_1) \cap \cdots \cap (X\setminus Q_s). \]
Thus, to show that $P \setminus Q \in \mf B'$, it suffices to show that, say, $P\cap (X\setminus Q_1)\in\mf B'$.  Let $Q_1 = \bigcap_{i=1}^n A_i \cap \bigcap_{j=1}^m (X\setminus B_j)$, so that
\begin{align*}
P \cap (X\setminus Q_1) &= P \cap \Big( \bigcup_{i=1}^n (X\setminus A_i) \cup \bigcup_{j=1}^m B_j \Big) \\
&= \bigcup_{i=1}^n (P\cap (X\setminus A_i)) \cup \bigcup_{j=1}^m (P\cap B_j) 
= \bigcup_{i=1}^n (P\setminus A_i) \cup \bigcup_{j=1}^m (P\cap B_j).
\end{align*}
As each $P\cap B_j\in\mf B'$ and $\mf B'$ is closed under unions, it remains to show that, say, $P\setminus A_1\in\mf B'$, but $P \setminus A_1 = \bigcup_{k=1}^t P_k \setminus A_1$, so in fact it remains to show that, say, $P_1 \setminus A_1\in\mf B'$.  If $P_1 = \bigcap_{i=1}^{n'} C_i \cap \bigcap_{j=1}^{m'} (X\setminus D_j)$ then
\[ P_1 \setminus A_1 = \bigcap_{i=1}^{n'} C_i \cap \bigcap_{j=1}^{m'} (X\setminus D_j) \cap (X\setminus A_1) \]
and so indeed $P_1 \setminus A_1\in\mf B'$.

Having shown that $\mf B = \mf B'$, it is easy that each member of $\mf B$ is a \emph{disjoint} union of sets of the form (1), because for any sets $(A_i)$ we have that $\bigcup_{i=1}^n A_i = A_1 \cup (A_2\setminus A_1) \cup (A_3\setminus (A_1\cup A_2)) \cup \cdots$.
\end{proof}

We shall use the following repeatedly, and so we pull it out as a lemma.

\begin{lemma}\label{lem:6}
Let $G$ be a group, and let $H_1,\cdots,H_n$ be subgroups.  There are subgroups $K_1,\cdots,K_{n'}$, each a subgroup of some $H_i$, and with $[K_i:K_i\cap K_j]=1$ or $\infty$ for all $i,j$, and with each $H_i$ a finite union of cosets of the $K_j$.  If the family of subgroups $\{H_i\}$ is closed under taking intersections, then the family $\{K_i\}$ is a subfamily of $\{H_i\}$.
\end{lemma}
\begin{proof}
If for some $i$, we have that $[H_1 : H_1 \cap H_i]$ is not $1$ or $\infty$, then $L = H_1\cap H_i$ is a subgroup of $H_1$ of finite index, and so $H_1$ can be covered by finitely many translates of $L$.  We can hence replace $H_1$ by $L$; notice that $[L : L\cap H_i]=1$.  We claim further that by replacing $H_1$ by $L$, if previously $[H_1:H_1\cap H_j]\in\{1,\infty\}$, then we do not change this property.  Indeed, if $[H_1:H_1\cap H_j] = 1$ then $H_1 = H_1 \cap H_j$ and so also $L\subseteq H_j$ so $[L:L\cap H_j]=1$.  If $[H_1:H_1\cap H_j] = \infty$, let then $K = H_1 \cap H_j$, so that $L \cap H_j = H_1 \cap H_i \cap H_j = L \cap K$, and hence $L\cap K \leq L \leq H_1$.  Towards a contradiction, suppose that $[L:L\cap H_j]<\infty$, so $[L:L\cap K] < \infty$.  As $L$ is of finite index in $H_1$, we can find $t_i\in H_1$ with $H_1 = \bigcup_{i=1}^m t_i L$, and we can find $r_j$ in $L$ with $L = \bigcup_{j=1}^{m'} r_j (L\cap K)$.  Thus
\[ H_1 = \bigcup_{i=1}^m t_i L = \bigcup_{i=1}^m \bigcup_{j=1}^{m'} t_i r_j (L\cap K), \]
and so certainly $H_1 = \bigcup_{i,j} t_i r_j K$ so $[H_1:H_1 \cap H_j] \leq mm' < \infty$, a contradiction.

By performing this argument finitely many times, we may suppose that $[H_1 : H_1 \cap H_i]\in \{1,\infty\}$ for each $i$.  We now look at $H_2$, and apply the same argument, and so forth.  This process could lead to repeats, and so possibly $n' \leq n$.

For the final remark, note that by construction, each $K_i$ is an intersection of some of the $H_j$, and thus $\{K_i\}$ is a subfamily of the $\{H_j\}$.
\end{proof}

The following result is used extensively in \cite{cohen, is}, and appears to have first been shown
in \cite{neu}.  Given the tools we now have, this is easy, so we give the proof.  This proof is essentially Cohen's from \cite{cohen}.

\begin{lemma}\label{lem:5}
Let $G$ be a group.  $G$ cannot be written as the finite union of cosets of infinite index.
\end{lemma}
\begin{proof}
Towards a contradiction, suppose that $G$ is the finite union of cosets of subgroups $H_1,\cdots,H_k$ where $[G:H_i]=\infty$ for each $i$.  By Lemma~\ref{lem:6} there are subgroups of these, say $K_1,\cdots,K_{k'}$ with 
$[K_i:K_i\cap K_j]=1$ or $\infty$ for all $i,j$, and still with each $K_i$ of infinite index in $G$.  As each $H_j$ is the finite union of cosets of some $K_i$, it follows that $G$ can be covered by finitely many cosets of the $K_i$, say $C_1,\cdots,C_n$.

As $K_1$ is of infinite index in $G$, there is some coset of $K_1$ which is not a member of our covering, say $K = sK_1$.  Then $K \cap K_i$ is a coset of $K_1 \cap K_i$, or is empty, for each $i$.  However, as $G$ is covered by the $C_i$, also $K$ is covered by $\{ K\cap C_i\}$.  If $C_i$ is a coset of $H_1$, then $K\cap C_i=\emptyset$, so we conclude that $K$ is covered by finitely many cosets of the subgroups $K_2, K_3,\cdots, K_{k'}$.  By translating by $s^{-1}$, also $K_1$ is covered by finitely many cosets of the subgroups $K_2, K_3,\cdots,K_{k'}$.

We now complete the argument by using induction on the number of subgroups, $k'$.  If we have only one subgroup, that is, $k'=1$, the result is trivially true.  The previous paragraph then provides the induction step.
\end{proof}

As the intersection of two cosets is again a coset, Lemma~\ref{lem:3} immediately implies that every member of $\Omega(G)$ is the finite disjoint union of sets of the form $E_0 \setminus \bigcup_{i=1}^n E_i$ where each $E_0$ is a coset, and by replacing $E_i$ with $E_0\cap E_i$, we may suppose further that $E_i\subseteq E_0$, for each $i$.  In fact, by using the arguments above, we can say more (again, see \cite{cohen, ilie}).

\begin{corollary}\label{corr:1}
Let $G$ be a group.  Every member of $\Omega(G)$ is either empty, or is a finite disjoint union of sets of the form
\begin{equation} E_0 \setminus \bigcup_{i=1}^n E_i, \label{eq:2} \end{equation}
where each $E_i$ is a coset, $E_i\subseteq E_0$, and $E_i$ has infinite index in $E_0$, for each $i>0$.
\end{corollary}
\begin{proof}
Let $Y \in \Omega(G)$, so in particular, $Y$ is in the ring generated by some subgroups $H_1,\cdots,H_n$, and their translates.  By adding in intersections, we may suppose that the finite family $\{H_i\}$ is closed under intersections.  By Lemma~\ref{lem:6} we may suppose that $[H_i:H_i \cap H_j]=1$ or $\infty$ for each $i,j$.  
By using this lemma, the resulting family $\{H_i\}$ may no longer be closed under intersections, but if $H$ is an intersection of some of the $H_i$, then $H$ is at least a finite union of cosets of some $H_j$.

By Lemma~\ref{lem:3}, $Y$ is the disjoint union of say $n$ sets of the form $E_0 \setminus \bigcup_{i=1}^m E_i$ where $E_0$ is a coset of some finite intersection of the $H_j$, and each $E_i$ is a coset of some $H_j$.  Hence $E_0$ is a finite union of cosets of the $H_j$, and so (by maybe increasing $n$) we may suppose that $E_0$ actually is a coset of some $H_j$.  As before, by replacing $E_i$ by $E_0\cap E_i$, we may suppose that $E_i$ is a subcoset of $E_0$.  As $[H_i:H_i \cap H_j]=\infty$ unless $H_i \subseteq H_j$, if $E_i$ has finite index in $E_0$ then $E_i$ and $E_0$ must be cosets of the same subgroup, and as $E_i \subseteq E_0$ we must have $E_0 = E_i$, a case which may be ignored.
\end{proof}

We now depart from the presentation of \cite{cohen}, and collect some further lemmas which will be used later.

\begin{lemma}\label{lem:7}
Let $G$ be a group, and let $H_1,\cdots,H_n$ be subgroups of infinite index in $G$.  Let $(K_i)_{i=1}^N$ be cosets of the $H_j$, and let $C = \bigcup_i K_i$.  There is $m$ and $t_1,\cdots,t_m\in G$ with $\bigcap_i t_i C = \emptyset$.
\end{lemma}
\begin{proof}
Notice that $\bigcap_{i=1}^m t_iC = \bigcap_{i=1}^m \bigcup_j t_iK_j = \bigcup\{ \bigcap_{i=1}^m t_i K_{j(i)} \}$ where the union is over all functions $j:\{1,\cdots,m\} \rightarrow \{1,\cdots,N\}$.  Thus we need to find $t_i$ so that $\bigcap_{i=1}^m t_i K_{j(i)} = \emptyset$ for any such function $j$.  In what follows, suppose that $K_j$ is the coset $s_j H_{k(j)}$, for each $j$.

Set $t_1=e,$ the identity.  We claim that there is $t_2$ with $K_j \cap t_2K_j=\emptyset$ for all $j$.  Indeed, if not, then for each $t_2$ there is $j$ with $s_j H_{k(j)} = t_2s_j H_{k(j)}$ (as cosets are either equal, or disjoint).  Equivalently, for each $t_2$ there is $j$ with $s_j^{-1} t_2 s_j \in H_{k(j)}$ so $t_2 \in s_j H_{k(j)} s_j^{-1}$.  As each subgroup $s_j H_{k(j)} s_j^{-1}$ is of infinite index in $G$, this shows that $G$ is covered by a finite union of subgroups of infinite index, a contradiction.

Suppose we have chosen $t_1,\cdots,t_p$ with the property that $t_1 K_{j(1)} \cap \cdots \cap t_p K_{j(p)}$ is only (possibly) non-empty when the $j(i)$ are all distinct.  We have already shown this is true for $p=2$.  To show that the claim holds for $p+1$, it suffices to find $t_{p+1}$ with $t_i K_j \cap t_{p+1} K_j = \emptyset$ for all $i\leq p$ and all $j$.  This is sufficient, for if we have $j(1),\cdots,j(p+1)$ not all distinct, say $j(k) = j(k')$ for $k < k'$, then if $k'\leq p$, we know already that $\bigcap_{i=1}^p t_i K_{j(i)} = \emptyset$, while if $k'=p+1$ then by construction $t_k K_{j(k)} \cap t_{p+1} K_{j(k)} = \emptyset$ so certainly $\bigcap_{i=1}^{p+1} t_i K_{j(i)} = \emptyset$.

To find $t_{p+1}$ we again proceed by contradiction, and suppose that for each $t\in G$ there are some $i,j$ with $t_i K_j = t K_j$ so that $t_i s_j H_{k(j)} = t s_j H_{k(j)}$ so $s_j^{-1} t_i^{-1} t s_j \in H_{k(j)}$.  That is, $t \in \bigcup_{i,j} t_i s_j H_{k(j)} s_j^{-1}$ which is again a finite union of cosets of infinite index, which cannot cover $G$.

So by induction, our claim holds for all $p$.  Thus $\bigcap_{i=1}^{p} t_i K_{j(i)}$ can only possibly be non-empty when the $j(i)$ are all distinct, but there are only $N$ many choices, and so if $p>N$ we have shown our claim.
\end{proof}

\begin{lemma}\label{lem:8}
Let $G$ be a group, $H$ a subgroup, and $H_1,\cdots,H_n$ subgroups with $[H:H\cap H_i]=\infty$ for each $i$.  Let $C$ be a union of finitely many cosets of the $H_i$.  There are $m$ and $t_1,\cdots,t_m\in H$ with $\bigcap_i Ct_i = \emptyset$.
\end{lemma}
\begin{proof}
The proof is similar to the previous one; but note here we need to choose $t_i$ in $H$ not $G$.  Let our cosets be $C_j$ for $j=1,\cdots,N$.  Again, we need to find $t_i$ with $C_{j(1)} t_1 \cap \cdots \cap C_{j(m)} t_m = \emptyset$ for any choices $j(i)$.  Set $t_1=e$.

Suppose we have $t_1,\cdots,t_p$ so that if $j(1),\cdots,j(p)$ are not distinct then $\bigcap_{i=1}^p C_{j(i)} t_i = \emptyset$.  Proceeding by induction, $t_{p+1} \in H$ needs to satisfy that $C_j t_i \cap C_j t_{p+1} = \emptyset$ for all $i,j$.  If no such $t_{p+1}$ exists, then for all $t\in H$ there are some $i,j$ with $C_j t_i \cap C_j t \not=\emptyset$.  If $C_j = sH_k$, say, then $s H_k t_i \cap s H_k t \not=\emptyset$, so there are $a,b\in H_k$ with $sat_i = sbt$ so $at_i = bt$ so $tt_i^{-1} = b^{-1} a \in H_k$ so $t \in H_k t_i$.  Thus $H \subseteq \bigcup_{i,k} H_k t_i$, but as each $t_i\in H$, we have that $H \cap H_kt_i = (H \cap H_k)t_i$ and so $H = \bigcup_{i,k} (H \cap H_k)t_i$.  As each $H\cap H_k$ is of infinite index in $H$, this is a contradiction.
\end{proof}

We now start on our proof of Theorem~\ref{thm:one}.  
We start with $Y\in\Omega(G)$, so there are subgroups $H_1,\cdots,H_n$ so that $Y$ is in the relative ring of sets generated by the $H_i$ (if necessary, we can choose one of the $H_i$ to be $G$).  We may suppose that the family $\{H_i\}$ is closed under intersection.  Then apply Lemma~\ref{lem:6} to suppose that $[H_i:H_i\cap H_j]=1$ or $\infty$ for all $i,j$, and that if $H$ is the intersection of some of the $H_i$, then $H$ is a finite disjoint union of cosets of some $H_j$.  Using Lemma~\ref{lem:3} as in the proof of Corollary~\ref{corr:1}, $Y$ is a finite union of sets $L_1,\cdots,L_m$ where
\begin{equation} L_i = E_0^{(i)} \setminus \bigcup_{j=1}^{n_i} E^{(i)}_j, \label{eq:4}
\end{equation}
where each $E^{(i)}_j$ is a coset of some $H_k$, and $E^{(i)}_j$ is a subcoset of infinite index in $E^{(i)}_0$, for each $i$ and $j>0$.

Our proof will be an induction, with the base case provided by the following proposition.  

\begin{proposition}\label{prop:3}
With notation as above, if $H_i$ is not contained in any other $H_j$, then there is a subgroup $H$ which contains $H_i$ as a finite-index subgroup (possibly $H=H_i$), such that $H \in \mc R_2(\{Y\})$, and such that $Y \in \mc R(\{H\} \cup \{ H_j :  j\not=i \})$.
\end{proposition}

Given this, we can now prove Theorem~\ref{thm:one}.

\begin{proof}[Proof of Theorem~\ref{thm:one}]
We can form a directed acyclic graph (DAG, see \cite{hgt} for example) with vertices the subgroups $H_i$, and with a directed edge from $H_i$ to $H_j$ when $H_i \supseteq H_j$ and there is no other $H_k$ with $H_i \supseteq H_k \supseteq H_j$.  We shall say that $H_i$ is \emph{top level} if $H_i$ is not contained in any other $H_j$, that is, there is no edge into $H_i$.  The \emph{depth} of a DAG is the length of the longest directed path in the DAG.

For each top level $H_i$, let $H$ be given by Proposition~\ref{prop:3}.  For any $i\not=j$ as $[H_i : H_i\cap H_j]=1$ or $\infty$, also $[H:H\cap H_j]=1$ or $\infty$, see Lemma~\ref{lem:12} below.  Also, if $H_i \cap H_j$ is the finite union of cosets of $H_k$, then so is $H\cap H_j$, see Lemma~\ref{lem:12} below.  Hence we may replace $H_i$ by $H$ and not change any of our assumptions.  Do this for all top level $H_i$, so that $Y\in\mc R(\{H_i\})$ and each top-level $H_i$ is in $\mc R_2(\{Y\})$.

We shall give a proof by induction on the depth of the DAG.  Notice that Proposition~\ref{prop:3}, and the previous paragraph, shows that the result is true when the DAG has depth $0$ (that is, all subgroups are top level).

Let $H$ be some top level $H_i$.  Let $Y = \bigsqcup_i L_i$ as before (see (\ref{eq:4})), and reorder these so that $E^{(i)}_0$ is a coset of $H$ for $i\leq n_0$, and not for $i>n_0$.  Define
\[ Y_0 = \bigcup_{i=1}^{n_0} E^{(i)}_0 \setminus Y \subseteq \bigcup_{i=1}^{n_0} E^{(i)}_0. \]
From the form that each $L_i$ is written in, namely that $E^{(i)}_j$ is a coset of infinite index
in $E^{(i)}_0$, it is clear that $E^{(i)}_j$ is not a coset of $H$ for $j\geq 1$.  From this, it follows that $Y_0 \in \mc R(\{ H \cap H_i\} \setminus \{H\})$.  For each $i$, either $H\cap H_i = H$ which we remove, or $H\cap H_i$ is a union of cosets of some $H_k$, where necessarily $H_k \subseteq H$.  Thus actually $Y_0 \in \mc R(\{ H_k : H_k \subsetneq H \})$.

Now the family $\{H_k : H_k \subsetneq H \}$ forms a subgraph of our DAG, in fact, it is DAG ``underneath'' $H$ (all the vertices which have a path leading to them from $H$).  Thus it is of smaller depth, and so by induction, there are subgroups $H_j'$ in $\mc R_2(\{Y_0\})$ with $Y_0 \in \mc R(\{H_j'\})$.  Notice that as $Y, H\in\mc R_2(\{Y\})$ also $Y_0 \in \mc R_2(\{Y\})$, and so also each $H_j' \in \mc R_2(\{Y\})$.

Next, reorder so that $H_i$ is top level for $i\leq n_0$, and not for $i>n_0$, so if $i>n_0$ the subgroup $H_i$ is contained is some $H_j$ with $j\leq n_0$.  For $i\leq n_0$, let $A_i$ be the union of the sets $E^{(k)}_0$ which are cosets of $H_i$.  Set $B_i = A_i \setminus Y$ and set $C = Y \setminus \bigcup_{i\leq n_0} A_i$.  We have shown that for each $i$, there are subgroups $H_j'$ in $\mc R(\{Y\})$ with $B_i \in \mc R(\{ H_j' \})$.

We now consider $C$.  As each $E^{(i)}_0$ is a coset of some $H_j$, either $E^{(i)}_0 \subseteq A_j$ for some $j$, or otherwise $E^{(i)}_0$ is a coset of some $H_j$ which is not top level.  As each $E^{(i)}_j$ is a subcoset of $E^{(i)}_0$, we see that $C$ is contained in $\bigcup \{ E^{(i)}_0 : E^{(i)}_0 \text{ coset of some }H_j\text{ which is not top level}\}$, and so $C \in \mc R(\{H_i : H_i \text{ not top level}\})$.  The DAG given by removing all top level subgroups has smaller depth, and so again by induction, there are subgroups $H_j''$ in $\mc R_2(\{C\})$ with $C\in\mc R(\{H_j''\})$.  Notice that $C \in \mc R_2(\{Y, A_i\}) \subseteq \mc R_2(\{Y, H_i : i\leq n_0\}) = \mc R_2(\{Y\})$ and so each $H_j''\in\mc R_2(\{Y\})$.

Let $\alpha$ be the collection of all the $H_j'$, the $H_j''$, and the top level $H_i$, so each subgroup in $\alpha$ is in $\mc R_2(\{Y\})$.  Then $A_i, B_i, C\in \mc R(\alpha)$.  As $A_i \setminus B_i = A_i \cap Y$, we see that
\[ C \cup \bigcup_{i\leq n_0} (A_i\setminus B_i) =
\Big(Y \setminus \bigcup_{i\leq n_0} A_i \Big) \cup \bigcup_{i\leq n_0} \big(A_i \cap Y\big) = Y. \]
Thus also $Y\in\mc R(\alpha)$, which completes the proof.
\end{proof}

Thus it remains to show Proposition~\ref{prop:3}.

\begin{definition}
Given the subgroups $(H_i)$, and $A\subseteq G$ some subset, we shall say that a coset $sH_j$ is \emph{big in $A$} if $A \cap sH_j$ cannot be covered by finitely many cosets of subgroups in $\{ H_i : i\not=j \}$.
\end{definition}

Let us make some easy remarks about this definition, which we put into a lemma for future reference.

\begin{lemma}\label{lem:10}
Let $(H_i)$ be subgroups as above.
\begin{enumerate}
\item\label{lem:10:1} If $A\subseteq B$ and $sH_j$ is big in $A$, then it is big in $B$.
\item\label{lem:10:2} With $A=\bigcup_{i=1}^n A_i$, we have that $sH_j$ is big in $A$ if and only if it is big in some $A_i$.
\end{enumerate}
\end{lemma}
\begin{proof}
(\ref{lem:10:1}) is clear.  For (\ref{lem:10:2}), if $sH_j$ is not big in any $A_i$, so $sH_j \cap A_i$ can be covered, and hence so can $A$ (as $A$ is a finite union); the converse follows as $A_i \subseteq A$ for each $i$.
\end{proof}

In the following results, we state the result for the subgroup $H_1$, but this is merely for notational convenience, as clearly there is nothing special about $H_1$ as compared to any other $H_i$.

\begin{lemma}\label{lem:11}
With $Y = \bigsqcup_i L_i$ as before, we have that $sH_1$ is big in $Y$ if and only if some $L_i$ is of the form $E_0 \setminus \bigcup_j E_j$ with $E_0 = sH_1$.  Furthermore, in this case the choice of $i$ is unique.
\end{lemma}
\begin{proof}
If $L_i = sH_1 \setminus \bigcup E_j$ and yet $sH_1$ is not big in $L_i$ then $sH_1 \cap L_i = L_i$ can be covered by finitely many cosets of infinite index in $H_1$.  Union these cosets with the $\{ E_j : j>0 \}$ and we have covered all of $sH_1$ by finitely many cosets of infinite index in $H_1$, a contradiction.  So $sH_1$ is big in $L_i$ and hence big in $Y$, by Lemma~\ref{lem:10}(\ref{lem:10:1}).

If $sH_1$ is big in $Y$ then by Lemma~\ref{lem:10}(\ref{lem:10:2}), $sH_1$ is big in some $L_i$.  If $L_i = E_0 \setminus \bigcup_j E_j$, then towards a contradiction, suppose that $E_0$ is not $sH_1$.  If $E_0$ is some other coset of $H_1$, then $sH_1 \cap E_0 = \emptyset$, so certainly $sH_1$ is not big in $L_i$.  So $E_0$ is a coset of some other subgroup $H_j$, and so $sH_1 \cap H_j$ is either empty (again, not possible) or is a coset of $H_1 \cap H_j$ which has infinite index in $H_1$.  Then $sH_1 \cap L_i$ is contained in a coset of infinite index in $sH_1$, and so is covered, so $sH_1$ not big in $L_i$, contradiction.

To show uniqueness, let each $L_i$ have the form of equation~\ref{eq:4}.  If $E^{(i)}_0 = E^{(j)}_0 = sH_1$, then if $i\not=j$, then as $L_i$ and $L_j$ are disjoint, we must have that $\bigcup_{k\geq 1} E^{(i)}_k \cup \bigcup_{k\geq 1} E^{(j)}_k$ is all of $sH_1$, a contradiction as these are cosets of infinite index.
\end{proof}

Let $C$ be the union of all $E^{(i)}_0$ which are cosets of $H_1$, so by the lemma, $C$ is the union of all cosets of $H_1$ which are big in $Y$.  Define
\[ \mc B = \Big\{ \bigcup_{i=1}^n s_i H_1 : \exists\, A\in\mc R(\{Y\}) \text{ so that }
sH_1 \text{ is big in } A \text{ if and only if } sH_1 = s_iH_1 \text{ for some }i \Big\}. \]
That is, $\mc B$ is the collection of sets $B$ which are finite unions of cosets of $H_1$, with the given property: there is $A\in\mc R(\{Y\})$ such that a coset $sH_1$ is big in $A$ if and only if $sH_1\subseteq B$.

\begin{lemma}\label{lem:9}
$\mc B = \mc R(\{C\})$.
\end{lemma}
\begin{proof}
To get a handle on $\mc B$, we need some information about $\mc R(\{Y\})$.  By Lemma~\ref{lem:3}, every $A\in\mc R(\{Y\})$ is of the form
\[ A = \Big( \bigcap_{i=1}^n s_i Y \Big) \setminus \Big( \bigcup_{j=1}^m t_j Y \Big), \]
where $n\geq 1$ and $s_i,t_j\in G$.  (This follows as if $\mf A$ is left-invariant, so will be $\mc R(\mf A)$.)  We wish to know when $sH_1$ is big in $A$, in terms of the form $Y = \bigsqcup L_i$ as in equation (\ref{eq:4}).

Let us think about these two parts individually.  Consider
\[ Y \cap tY = \bigsqcup_{i,j} L_i \cap tL_j
= \bigsqcup_{i,j}\Big(\big(E^{(i)}_0 \cap tE^{(j)}_0\big) \setminus \big(\bigcup_k E^{(i)}_k \cup \bigcup_l tE^{(j)}_l \big) \Big). \]
As argued in the ``uniqueness'' part of Lemma~\ref{lem:11}, either $E^{(i)}_0 = tE^{(j)}_0 = sH_1$, or $E^{(i)}_0 \cap tE^{(j)}_0 \cap sH_1$ is either empty, or is a coset of infinite index in $sH_1$.  It then follows from Lemma~\ref{lem:11} that $sH_1$ is big in $Y\cap tY$ if and only if there are (unique) $i,j$ with $E^{(i)}_0 = tE^{(j)}_0 = sH_1$.  A similar argument now shows that $sH_1$ is big in $\bigcap s_i Y$ if and only if, for each $i$ there is a (necessarily unique) $j$ with $s_i E^{(j)}_0 = sH_1$.

By Lemma~\ref{lem:10}(\ref{lem:10:2}) we see that $sH_1$ is big in $\bigcup t_j Y$ if and only if there is some $j$ with $sH_1$ big in $t_jY$, if and only if, by Lemma~\ref{lem:11}, there is $j$ and a (necessarily unique) $k$ with $sH_1 = t_j E^{(k)}_0$.

Let $B_0=\bigcap s_iY$ and $B_1=\bigcup t_jY$, so $A = B_0\setminus B_1$.  Then:
\begin{itemize}
\item if $sH_1$ is not big in $B_0$, then it is not big in $A$, by Lemma~\ref{lem:10}(\ref{lem:10:1}).
\item if $sH_1$ is big in $B_0$ and not big in $B_1$, then as $B_0 = A\cup B_1$, also $sH_1$ is big in $A$, by 
Lemma~\ref{lem:10}(\ref{lem:10:2}).
\item if $sH_1$ is big in both $B_0$ and $B_1$, then from above, we know that $sH_1 \cap B_0$ is equal to $sH_1$ with a finite union of cosets of infinite index removed, while $sH_1 \cap B_1$ contains a set of the form $sH_1$ with a finite union of cosets of infinite index removed.  Thus $sH_1 \cap (B_0\setminus B_1)$ is contained in a finite union of cosets of infinite index, so $sH_1$ not big in $A = B_0\setminus B_1$.
\end{itemize}
In conclusion, $sH_1$ is big in $A$ if and only if $sH_1$ is big in $B_0$ but not big in $B_1$.  Now, $sH_1$ is big in $B_0$ if and only if $sH_1 \subseteq s_i C$ for all $i$, that is, $sH_1 \subseteq \bigcap s_i C$.  Also $sH_1$ is big in $B_1$ exactly when $sH_1 \subseteq t_jC$ for some $j$, that is, $sH_1 \subseteq \bigcup t_j C$.  Thus
\begin{equation} \bigcup \{ sH_1 : sH_1 \text{ big in A}\} = \bigcap s_i C \setminus \bigcup t_j C,
\label{eq:5}
\end{equation}
as all these sets are unions of cosets of $H_1$.  This shows that $\mc B \subseteq \mc R(\{C\})$.

To show the converse, we simply observe that by Lemma~\ref{lem:3} any member of $\mc R(\{C\})$ is of the form given by the right-hand-side of equation~\ref{eq:5}, for some $(s_i)$ and $(t_j)$, and so the associated $A$ will be a member of $\mc R(\{Y\})$, showing that $\mc R(\{C\}) \subseteq \mc B$.
\end{proof}

\begin{proposition}\label{prop:5}
There is a subgroup $H$, which is in $\mc B$ and so a finite union of cosets of $H_1$, such that any element of $\mc B$ is a finite union of cosets of $H$.
\end{proposition}
\begin{proof}
Every member of $\mc B$ is a finite union of cosets of $H_1$.  Let
\[ H = s_1 H_1 \sqcup \cdots \sqcup s_n H_1 \in \mc B \]
be chosen with $n>0$ minimal (so $H$ is the disjoint union of $n$ cosets of $H_1$, and any member of $\mc B$ is the disjoint union of at least $n$ cosets of $H_1$).  By translating, we may suppose that $s_1=e$ the identity.  We will show that $H$ is a subgroup.

If $A = \bigsqcup t_i H_1 \in \mc B$ then $A\cap H\in\mc B$, and so is either empty, or the union of at least $n$ cosets of $H_1$.  As $A \cap H \subseteq H$ we must have $H=A\cap H$ or $A\cap H=\emptyset$.  In particular, for each $s$, either $sH \cap H = H$ or $sH \cap H=\emptyset$.  If $s\in H$ then as $e\in H$ also $s \in sH$ so $s\in H \cap sH$ so $H\cap sH = H$.  Thus $H H \subseteq H$.  Also $e\in H \cap s^{-1}H$ and so $H \cap s^{-1}H = H$ so in particular $s^{-1} = s^{-1}e \in s^{-1}H \subseteq H$, and we conclude that $H$ is a subgroup.

Given $A\in\mc B$ and $s\in G$ notice that $sH \cap A = s(H\cap s^{-1}A)$ is either empty or equal to $sH$, because $s^{-1}A\in\mc B$.  It follows that $A$ is a (necessarily finite) union of cosets of $H$.
\end{proof}

\begin{lemma}\label{lem:12}
Let $H$ be a subgroup containing $H_1$ with $[H:H_1] < \infty$.  For each $i>1$
we have that $[H:H\cap H_i] = \infty$ or $1$, and that $H\cap H_i$ is a finite union of cosets of some $H_k$.
\end{lemma}
\begin{proof}
For $i>1$, we now that either $[H_1:H_1\cap H_i]=1$ or $\infty$.
If $[H_1:H_1\cap H_i]=1$ then $H_i\subseteq H_1 \subseteq H$.  Otherwise, $[H_1:H_1\cap H_i]=\infty$, and we claim that also $[H:H\cap H_i] = \infty$.  If not, then $H = \bigcup_{i=1}^n s_i(H\cap H_i)$ say, and as $H_1\leq H$, it follows that $H_1 = H_1 \cap H = \bigcup H_1 \cap s_i(H \cap H_i)$.  For each $i$, either $H_1 \cap s_i(H\cap H_i)$ is empty, or is a coset of $H_1 \cap (H\cap H_i) = H_1 \cap H_i$, and so $[H_1:H_1\cap H_i]\leq n$, a contradiction.

Let $H_1 \cap H_i$ be a finite union of cosets of some $H_k$, this being one of our properties of the family $\{H_i\}$.  As $H_1$ is finite index in $H$, we have that $H = \bigcup_{j=1}^n s_j H_1$ say.  Then $H \cap H_i = \bigcup_j s_jH_1 \cap H_i$, and for each $j$ either $s_j H_1 \cap H_i$ is empty, or is a coset of $H_1 \cap H_i$, which is a finite union of cosets of $H_k$.  Thus $H\cap H_i$ is also a finite union of cosets of $H_k$.
\end{proof}

We can now complete the proof of Proposition~\ref{prop:3}.

\begin{proof}[Proof of Proposition~\ref{prop:3}]
Our hypothesis is that $H_i$ is not contained in any other $H_j$.  We have been
supposing, by reordering, that $i=1$.  Let $H$ be given by Proposition~\ref{prop:5}.  We need to show that $H\in\mc R_2(\{Y\})$ and that $Y\in\mc R(\{H\} \cup \{ H_j : j\not=1\} )$.

As $H \in \mc B$ there is some $A_0\in\mc R(\{Y\}) \subseteq \mc R(\{H_j\})$ so that $sH_1$ is big in $A_0$ if and only if $sH_1\subseteq H$.  Then $A_0$ has the form
\begin{equation} A_0 = \bigsqcup_i\Big( F^{(i)}_0 \setminus \bigcup_j F^{(i)}_j\Big), \label{eq:3} \end{equation}
where each $F^{(i)}_j$ is some coset of some $H_k$, and each $F^{(i)}_j$, $j>0$, is a subcoset of infinite index in $F^{(i)}_0$.  Then $sH_1$ is big in $A_0$ if and only if $sH_1 = F^{(i)}_0$ for some $i$.  It follows that $H\setminus A_0$ is contained in the union of: (1) $F^{(i)}_j$ where $F^{(i)}_0$ is a coset of $H_1$; and (2) $F^{(i)}_0$ a coset of some $H_j$ for $j>1$.  So $H\setminus A_0$ is contained in a finite union of cosets of $\{ H\cap H_j : j>1 \}$.  By Lemma~\ref{lem:12}, $H\cap H_j$ is of infinite index in $H$ for $j>1$, and so we can apply Lemma~\ref{lem:7} to $H$ to find $(t_i)$ in $H$ with
\[ \emptyset = \bigcap_i t_i(H\setminus A_0) = H \setminus \bigcup_i t_i A_0. \]
So if we set $B_0 = \bigcup_i t_i A_0$ then $H\subseteq B_0$, and $B_0\in\mc R(\{Y\})$.

We claim that $sH_1$ is big in $B_0$ if and only if $sH_1\subseteq H$, which is equivalent to $s\in H$.  As $H\subseteq B_0$, ``if'' is clear.  To show ``only if'', suppose $sH_1$ is big in $B_0$, so by Lemma~\ref{lem:10}(\ref{lem:10:2}) $sH_1$ is big in $t_iA_0$ for some $i$, so $t_i^{-1}sH_1$ is big in $A_0$, so $t_i^{-1}sH_1\subseteq H$, so $sH_1\subseteq t_iH = H$.

$B_0$ has the same form of equation~(\ref{eq:3}), so again $sH_1\subseteq H$ if and only if some $F^{(i)}_0$ is equal to $sH_1$.  Thus, again, we see that $B_0$ is contained in the union of $H$ and cosets of $\{ H_j : j>1 \}$, say $B_0 \subseteq H \cup \bigcup_{i=1}^N A_i$, so each $A_i$ is a coset of some $H_j, j>1$.  With $C = \bigcup_i A_i$, by Lemma~\ref{lem:8} there are $t_1,\cdots,t_m\in H$ with $\bigcap_i Ct_i = \emptyset$.  As $B_0 \subseteq H \cup C$,
\[ \bigcap_i B_0 t_i \subseteq \bigcap_i \big( H \cup C \big)t_i
= \bigcap_i \big( H \cup Ct_i \big)
= H \cup \bigcap_i Ct_i = H. \]
Thus clearly $H = \bigcap_i B_0t_i$ and so $H\in\mc R_2(\{Y\})$.

To finish the proof of Proposition~\ref{prop:3}, it remains to show that $Y \in \mc R(\{H\} \cup \{ H_j :  j\not=1 \})$.  By Lemma~\ref{lem:9} we have that $\bigcup\{ E^{(i)}_0 : E^{(i)}_0 \text{ a coset of }H_1\}$ is a union of cosets of $H$.  Reorder so that $E^{(1)}_0 \cup\cdots\cup E^{(k)}_0 = sH$, say, so that
\[ \bigcup_{i=1}^k E^{(i)}_0 \setminus \bigcup_j E^{(i)}_j
= sH \setminus \bigcup_{i=1}^k \bigcup_j E^{(i)}_j. \]
Notice that each $E^{(i)}_j$ is a coset of some $H_t$ for some $t>1$.  As $H_1$ is not contained in any other $H_k$, no $E^{(i)}_j$, with $j>0$, is a coset of $H_1$.  Thus in this way we can replace every usage of a coset of $H_1$ by a coset of $H$, so proving the claim, and completing the proof.
\end{proof}

\begin{remark}\label{rem:1}
It was only in the final step of the previous proof that we started to work with right translations, as well as left translations.  This seems necessary, as the following example shows.  Consider $\mathbb F_2$ the free group with generators $a,b$, and set $H_1 = \langle a \rangle = \{ a^n : n\in\mathbb Z \}$ and $H_2 = b H_1 b^{-1}$.  Then $H_1, H_2$ are subgroups and $H_1 \cap H_2 = \{e\}$, so the family $\{H_1,H_2\}$ satisfies our assumptions.

Now let $Y = H_1 \sqcup b^{-1} H_2$ which is the disjoint union of cosets of the $H_i$.  Then $Y = \langle a \rangle \sqcup \langle a \rangle b^{-1}$.  For $x\in\mathbb F_2$, let $x = ya^n$ where $y$ is a reduced word in $a,b$ which does not end in $a,a^{-1}$, and $n\in\mathbb Z$, so that $x\langle a\rangle = y\langle a\rangle$.  Thus $xY = y\langle a\rangle \sqcup y\langle a\rangle b^{-1}$ and so $xY$ is either equal to $Y$ or disjoint from $Y$.

We conclude that $\mc R(\{Y\})$ just consists of finite disjoint unions of left translates of $Y$.  In particular, neither $H_1$ nor $H_2$ is in $\mc R(\{Y\})$.
\end{remark}

\section{Application to completely bounded maps}

We now use Theorem~\ref{thm:one} to give a new proof of the main result of \cite{is}.  We are now following the end of Section~4 of \cite{cohen} fairly closely, but again with more details provided, and changes made from the abelian setting.

\begin{lemma}\label{lem:1}
Let $G,H$ be topological spaces, $Y\subseteq H$ a subset, and $\alpha:Y\rightarrow G$ a map which is continuous when $Y$ has the subspace topology.  Suppose there is a continuous map $\overline\alpha:\overline Y\rightarrow G$ extending $\alpha$.  Let $\Gamma = \mc G(\alpha)\subseteq H\times G$, the graph of $\alpha$.  If $\Gamma=U \cap C$ for some open $U$ and closed $C$ in $H\times G$, then $Y = \overline Y \cap V$ for some open $V\subseteq H$.  In particular, $Y$ is Borel.
\end{lemma}
\begin{proof}
As $\Gamma\subseteq C$ also $\overline\Gamma\subseteq C$, and as $\Gamma\subseteq U$ also $\Gamma \subseteq \overline\Gamma \cap U \subseteq C\cap U = \Gamma$, so we conclude that $\Gamma = \overline\Gamma \cap U$.

We claim that $\mc G(\overline\alpha) = \overline\Gamma$, which follows by
continuity of $\overline\alpha$.  Indeed, given $y\in\overline Y$ let $(y_i)$ be a net in $Y$ converging to $y$, so that $\alpha(y_i) = \overline\alpha(y_i) \rightarrow \overline\alpha(y)$ by continuity, and hence $(y_i, \alpha(y_i)) \rightarrow (y,\overline\alpha(y))$.  Hence $\mc G(\overline\alpha) \subseteq \overline\Gamma$.  Given $(y,x)\in\overline\Gamma$ let $(y_i, \alpha(y_i))$ be a net in $\Gamma$ converging to $(y,x)$, so $y_i\rightarrow y$, $\alpha(y_i)\rightarrow x$, so continuity of $\overline\alpha$ ensures that $\overline\alpha(y) = x$.  Thus $\overline\Gamma \subseteq \mc G(\overline\alpha)$.

Let $V_0 = \{ y\in\overline Y : (y,\overline\alpha(y))\in U \}$.  Given $y\in V_0$, as $U\subseteq H\times G$ is open, by the definition of the product topology, there is $W_1\subseteq H$ and $W_2\subseteq G$ both open, with $y\in W_1, \overline\alpha(y)\in W_2$ and $W_1\times W_2 \subseteq U$.  Then $\overline\alpha^{-1}(W_2)$ is open in $\overline Y$, so there is $W_0\subseteq H$ open with $W_0 \cap \overline Y = \overline\alpha^{-1}(W_2)$.  Given $x\in W_0 \cap W_1 \cap \overline Y$, we have that $x\in\overline\alpha^{-1}(W_2)$, so $\overline\alpha(x)\in W_2$, and $x\in W_1$, so $(x,\overline\alpha(x))\in U$, so $x\in V_0$.  Also $y\in W_0 \cap W_1 \cap \overline Y$.  Thus $V_0$ is open in $\overline Y$, as we have shown that each point $y$ has an open neighbourhood, namely $W_0 \cap W_1 \cap \overline Y$.

Let $V\subseteq H$ be open with $V\cap \overline Y = V_0$.  Given $y\in V_0$, so $(y,\overline\alpha(y))\in U$, also $(y,\overline\alpha(y))\in\mc G(\overline\alpha) = \overline\Gamma$, and so $(y,\overline\alpha(y)) \in U \cap \overline\Gamma = \Gamma$, and so $y\in Y$.  Conversely, if $y\in Y$ then $(y,\overline\alpha(y)) = (y,\alpha(y)) \in \Gamma = \overline\Gamma \cap U$, and so $y\in V_0$.  Thus $\overline Y \cap V = Y$ as required.
\end{proof}

Proposition~\ref{prop:1} below is implicitly assumed in the proof of \cite[Lemma~1.3(ii)]{is}, but we do not see why it follows immediately ``by uniformity of the topology''; compare the argument on \cite[page~223]{cohen}, which we follow.

\begin{lemma}\label{lem:4}
Let $L = E_0 \setminus \bigcup_{k=1}^n E_k$ be of the form of equation~\ref{eq:2}, so $E_0$ is a coset and each $E_k$ is a subcoset of infinite index.  For any $N$ there are $a_1,\cdots,a_N \in L$ with, for $i\not= j$, $a_i^{-1}a_j\not\in E_k^{-1}E_k$ for any $k\geq 1$.
\end{lemma}
\begin{proof}
We first show that this is true for $N=2$.  If not, then for all $a,b\in L$ there is some $k$
with $a^{-1}b\in E_k^{-1}E_k$.  That is, for each $a\in L$, we have $L \subseteq \bigcup_{k=1}^n a E_k^{-1} E_k$.
As $E_k^{-1} E_k$ is the subgroup which $E_k$ is a coset of, this shows that $L$ is contained
in the finite union of cosets of infinite index, and so $E_0$ is contained in some finite union
of cosets of infinite index, contradiction.

We now proceed by induction.  Suppose the claim holds for $N\geq 2$, but does not hold for $N+1$.
Then given any $a_1,\cdots,a_N$ satisfying the claim, we cannot find $a_{N+1}$ satisfying the
claim, so that for any $b\in L$ we have that $a_i^{-1}b \in E_k^{-1}E_k$ for some $i,k$.
That is,
\[ L \subseteq \bigcup_{i=1}^N \bigcup_{k=1}^n a_i E_k^{-1}E_k. \]
Again, it follows from this that $E_0$ is contained in some finite union
of cosets of infinite index, which is a contradiction.
\end{proof}

\begin{proposition}\label{prop:1}
Let $G,H$ be locally compact groups, and let $L = E_0 \setminus \bigcup_{k=1}^n E_k \subseteq H$ be of the form as in equation~\ref{eq:2}.  Let $\alpha:L\rightarrow G$ be a map which is continuous on $L$ for the subspace topology, and is the restriction of some affine map $\psi:E_0\rightarrow G$.  Then $\psi$ is continuous.
\end{proposition}
\begin{proof}
Choose $a_1,\cdots,a_{n+1}\in L$ using the lemma.  We claim that for any $x,y\in E_0$ there is
$k$ with $x y^{-1} a_k \in L$.  Indeed, if not, then as $x y^{-1} a_k\in E_0$ we must have that
$x y^{-1} a_k\in \bigcup_{i=1}^n E_i$, for each $k$.  By the pigeonhole principle, there is $i$
and $1\leq r<s\leq n+1$ with $x y^{-1} a_r, x y^{-1} a_s \in E_i$.  Thus $(x y^{-1} a_r)^{-1}
xy^{-1}a_s = a_r^{-1} a_s \in E_i^{-1} E_i$, contradiction.

Let $U$ be an open neighbourhood of $e$ in $G$, and choose an open symmetric neighbourhood $V$
of $e$ with $VV\subseteq U$.  Then $V \alpha(a_k)$ is an open neighbourhood of $\alpha(a_k)$
and so by continuity of $\alpha$, and the definition of the subspace topology,
there is an open $V_k$ in $H$ with
\[ V_k \cap L = \alpha^{-1}( V \alpha(a_k) ). \]
Then $a_k \in V_k\cap L$, and given $x,y\in V_k\cap L$ we see that $\alpha(x) = v_0\alpha(a_k)$
and $\alpha(y) = v_1\alpha(a_k)$, for some $v_0,v_1\in V$.  Then
\[ \alpha(x) \alpha(y)^{-1} = v_0\alpha(a_k) \alpha(a_k)^{-1} v_1^{-1}
= v_0 v_1^{-1} \in VV \subseteq U. \]

Now set $V_0 = \bigcap_{k=1}^m V_k a_k^{-1}$ an open neighbourhood of $e$ in $H$.
Let $x,y\in E_0$ with $x y^{-1} \in V_0$.  Then
there is $k$ with $x y^{-1} a_k \in L$, as above.  Also $x y^{-1} a_k \in V_0 a_k \subseteq V_k$
and $a_k \in V_0 a_k \subseteq V_k$.  As also $a_k\in L$, we conclude that
$\alpha(x y^{-1} a_k) \alpha(a_k)^{-1} \in U$.
Choose any $z\in E_0$.  Then $a_k = z z^{-1} a_k$.
As $\psi$ is affine, and agrees with $\alpha$ on $L$, we see that
\begin{align*} \psi(x) \psi(y)^{-1} &=
\psi(x) \psi(y)^{-1} \psi(a_k) \psi(a_k)^{-1} \psi(z) \psi(z)^{-1} \\
&= \psi(x) \psi(y)^{-1} \psi(a_k) \big( \psi(z) \psi(z)^{-1} \psi(a_k) \big)^{-1} \\
&= \psi(x y^{-1} a_k) \psi(z z^{-1} a_k)^{-1} =
 \alpha(x y^{-1} a_k) \alpha(a_k)^{-1} \in U.
\end{align*}

It follows that $\psi$ is (uniformly) continuous.  Indeed, if $x\in E_0$ and $(x_i)$ is a net in
$E_0$ converging to $x$, then $x_i^{-1} x \rightarrow e$ in $H$.  Given any neighbourhood $U$
of $e$ in $G$, pick $V_0$ as above, and observe that eventually $x_i^{-1}x\in V_0$.  Thus
$\psi(x_i^{-1}x) = \psi(x_i)^{-1}\psi(x) \in U$.  This shows that $\psi(x_i)^{-1}\psi(x)
\rightarrow e$ in $G$ so that $\psi(x_i)\rightarrow \psi(x)$.
\end{proof}

We are now in a position to complete our argument.  We recall the setup from the start of Section~\ref{sec:2}, so $\Phi:A(G)\rightarrow B(H)$ is a completely bounded homomorphism, and with $G$ amenable, and there is $\alpha:H\rightarrow G_\infty$ continuous, which implements $\Phi$.  With $Y=\alpha^{-1}(G)$ we additionally know that $\alpha:Y\subseteq H\rightarrow G$ is piecewise affine.

\begin{proposition}\label{prop:2}
With $\alpha:Y\subseteq H\rightarrow G$ as above, we have that $Y$ is clopen.  Furthermore, $Y$ may be written as the disjoint union of sets $L_i$ of the form in equation~\ref{eq:2}, say
\[ L_i = E^{(i)}_0 \setminus \bigcup_j E^{(i)}_j, \]
with each $E^{(i)}_j$ Borel, and for each $i$ there is a continuous affine map $\overline\alpha_i:
\overline{E^{(i)}_0} \rightarrow G$ which restricts to $\alpha$ on $L_i$.
\end{proposition}
\begin{proof}
We can write $Y$ as the disjoint union of sets of the form as in equation~\ref{eq:2}, say $Y = \bigsqcup_i L_i$, and for each $i$ there is an affine map $\alpha_i$ which agrees with $\alpha$ on $L_i$.  By Proposition~\ref{prop:1} we know that $\alpha_i$ is continuous, and as $\alpha_i$ is affine, it admits a (unique) continuous extension to the closure of the coset on which it is defined, compare \cite[Lemma~1.3(i)]{is}.

We claim that $\overline{L_i} \subseteq Y$.  Indeed, given $x\in\overline{L_i}$, there is a net $(x_j)$ in $L_i$ which converges to $x$, and so $\lim_j \alpha_i(x_j) = \overline\alpha_i(x)$, and as $\alpha_i$ extends $\alpha$ also $\lim_j \alpha(x_j) = \overline\alpha_i(x)$.  As $\alpha:H\rightarrow G_\infty$ is continuous, we conclude that $\alpha(x) = \overline\alpha_i(x) \in G$, and so $x\in Y$ (that is, $x$ is not the point $\infty$).  Notice that we have also shown that $\alpha$ agrees with $\overline\alpha_i$ on $\overline{L_i}$.
Thus we have that
\[ Y = \bigcup_{i=1}^n L_i \subseteq \bigcup_{i=1}^n \overline{L_i} \subseteq Y, \]
and so we have equality throughout.  Hence $Y$ is closed (and also open).

Now consider the graph $\Gamma=\mc G(\alpha) = \{ (y, \alpha(y)) : y\in Y \} \subseteq H\times G$.  That $\alpha$ is continuous shows that $\Gamma$ is closed, and that $\alpha$ is piecewise affine shows that $\Gamma \in \Omega(H\times G)$.  By Theorem~\ref{thm:one} there are subgroups $K_1,\cdots,K_n$ in $\mc R_2(\{\Gamma\})$ so that $\Gamma \in \mc R(\{K_1,\cdots,K_n\})$.  By Lemma~\ref{lem:3} we can write $\Gamma = \bigsqcup_i \Gamma_i$ where each $\Gamma_i$ is of the form (\ref{eq:2}), say $\Gamma_i = F^{(i)}_0 \setminus \bigcup_j F^{(i)}_j$, with each $F^{(i)}_j$ a coset of some $K_k$, and with $F^{(i)}_j \subseteq F^{(i)}_0$ for each $j$.
From Lemma~\ref{lem:3} we also know that each $K_k$ is of the form $C\cap U$ for some closed set $C$ and some open set $U$, because $\Gamma$ is closed, and so translates of $\Gamma$ are closed.  From Lemma~\ref{lem:3} once more, it follows that each member of $\mc R(\{K_i\})$ is also a finite union of sets of the form ``closed intersect open'', in particular this applies to each $F^{(i)}_j$.

The proof of \cite[Lemma~1.2(iii)]{is} (that is, Lemma~\ref{lem:2}) shows that each $F^{(i)}_0 \subseteq H\times G$ is the graph of an affine map, say $\phi_i : E^{(i)}_0\rightarrow G$, for some coset $E^{(i)}_0$ of $H$.  Then $F^{(i)}_j \subseteq F^{(i)}_0$ is also a graph, of the restriction of $\phi_i$ to a coset, say
$E^{(i)}_j$.  Thus, in the start of the proof, we could actually take $L_i = E^{(i)}_0 \setminus \bigcup_j E^{(i)}_j$ and $\alpha_i = \phi_i$.  In particular, each $\phi_i$ is continuous.  Applying Lemma~\ref{lem:1} to the restriction of $\phi_i$ to $E^{(i)}_j$ shows that $E^{(i)}_j$ is Borel (as $F^{(i)}_j$ is the graph).  
\end{proof}

We now suppose that $H$ is $\sigma$-compact.  Under this hypothesis, Steinhaus's Theorem, see \cite{str}, shows that if $C\subseteq H$ is a coset of non-zero (Haar) measure, then $C$ is open.  In the following proof, we use this to show that a finite family of \emph{Borel} cosets, each of which has empty interior, also has union with empty interior.  This is exactly the point which goes wrong in the attempted purely topological proof of \cite[Lemma~1.3]{is}.

\begin{proposition}\label{prop:4}
Let $H$ be $\sigma$-compact, and continue with the notation of Proposition~\ref{prop:2}.  There are $Y_1,\cdots,Y_m$ in the open coset ring of $H$ so that $Y$ is the disjoint union of the $Y_i$, and for each $i$ there is a continuous affine map $\alpha_i:\aff(Y_i)\rightarrow G$ which agrees with $\alpha$ on $Y_i$.
\end{proposition}
\begin{proof}
If $E^{(i)}_0$ is open, then as it is a coset, it must also be closed.  Reorder so that $E^{(i)}_j$ is open for $1\leq j\leq m$ and not open for $j>m$.  Let $Z_i = E^{(i)}_0 \setminus \bigcup_{j=1}^m E^{(i)}_j$.  As each $E^{(i)}_j$ is clopen, for $j\leq m$, it follows that $Z_i$ is clopen.

For $j>m$, as $E^{(i)}_j$ is not open, it has measure zero, and so also $\bigcup_{j>m} E^{(i)}_j$ has measure zero, and so $\bigcup_{j>m} E^{(i)}_j$ has empty interior.  As $L_i \subseteq Z_i$ it follows that $Z_i \setminus L_i \subseteq \bigcup_{j>m} E^{(i)}_j$ so $Z_i\setminus L_i$ has empty interior.  As $Z_i$ is open, we have shown that $Z_i \subseteq \overline{L_i}$.

However, $Z_i$ is closed, so as $L_i\subseteq Z_i$ also $\overline{L_i} \subseteq Z_i$.  We conclude that $\overline{L_i} = Z_i$ is clopen, and clearly in the open coset ring of $H$.

Now reorder so that $E^{(i)}_0$ is open for $i\leq m$ and not for $i>m$.  For $i>m$ we again have that $E^{(i)}_0$ has measure zero, so also $L_i$ has measure zero, so we again conclude that $\bigcup_{i>m} L_i$ has empty interior.  As $Y$ is clopen, and $\bigcup_{i=1}^m \overline{L_i}$ is clopen, it follows that $Y \setminus \bigcup_{i=1}^m \overline{L_i}$ is open.  As $\bigcup_{i=1}^m L_i \subseteq \bigcup_{i=1}^m \overline{L_i}$ it follows that $Y \setminus \bigcup_{i=1}^m \overline{L_i} \subseteq
Y \setminus \bigcup_{i=1}^m L_i = \bigcup_{i>m} L_i$ and so $Y \setminus \bigcup_{i=1}^m \overline{L_i}$ has empty interior, and hence must be empty.  So $Y \subseteq \bigcup_{i=1}^m \overline{L_i}$, but then $\bigcup_{i=1}^m \overline{L_i} \subseteq \overline{ \bigcup_{i=1}^m L_i } \subseteq \overline{Y} = Y$ so we conclude that $Y = \bigcup_{i=1}^m \overline{L_i}$.

Finally, we use that $\alpha_i$ extends continuously to $\overline\alpha_i$ a continuous affine map; restrict this to $\overline{L_i}$.  It seems possible that the $(\overline{L_i})$ are not disjoint, but if we replace $\overline{L_2}$ by $\overline{L_2} \setminus \overline{L_1}$, then we do not leave the open coset ring, and so we can simply adjust to obtain the disjoint family $(Y_i)$ as required.
\end{proof}

We can finally state and prove the main part of \cite[Theorem~3.7]{is}.

\begin{definition}
Let $G,H$ be locally compact groups.  A map $\alpha:Y\subseteq H\rightarrow G$ is a \emph{continuous piecewise affine} map when $Y$ is open, and can be written as a disjoint union $Y = \bigsqcup_i Y_i$ with each $Y_i\in\Omega_o(H)$, and for each $i$ there is an open coset $C_i$ and a continuous affine map $\alpha_i:C_i\rightarrow G$, with $Y_i\subseteq C_i$ and $\alpha_i$ agrees with $\alpha$ on $Y_i$.
\end{definition}

Notice that each $Y_i\in \Omega_o(H)$ is also closed, and so also $Y$ is closed; compare also Remark~\ref{rem:2} below.

\begin{theorem}\label{thm:2}
Let $\Phi:A(G)\rightarrow B(H)$ be a completely bounded homomorphism, with $G$ amenable.  There is a continuous piecewise affine map $\alpha:Y\subseteq H\rightarrow G$ with
\[ \Phi(u)(h) = \begin{cases} u(\alpha(h)) &: h\in Y \\ 0 &: h\not\in Y, \end{cases}
\qquad (u\in A(G), h\in H). \]
\end{theorem}
\begin{proof}
We have already proved this in the $\sigma$-compact case.  Now let $H$ be an arbitrary locally compact group.  With $L_i$ as in Proposition~\ref{prop:2}, we again wish to prove that $Y = \bigcup_{i=1}^m \overline{L_i}$ where $E^{(i)}_0$ is open for $i\leq m$, and not open otherwise.

There is $H_0 \leq H$ an open (and so closed) $\sigma$-compact subgroup.  Let $\beta$ be the restriction of $\alpha$ to $Y \cap H_0$.  We can apply Proposition~\ref{prop:4} to $\beta$, and so conclude that $H_0 \cap Y$ is the union of the sets $\overline{H_0\cap L_i}$ for those $i$ with $H_0 \cap E^{(i)}_0$ open and non-empty.  However, if $E^{(i)}_0$ is open, then also $H_0 \cap E^{(i)}_0$ is open, and if it is empty, there is no harm in considering it in the union.  Thus $H_0 \cap Y = \bigcup_{i=1}^m \overline{H_0 \cap L_i}$.  This argument would also apply to any translate of $Y$, equivalently, to any coset of $H_0$, so we conclude that $sH_0 \cap Y = \bigcup_{i=1}^m \overline{sH_0 \cap L_i}$ for any $s$.  As each $sH_0$ is clopen, it follows that $Y = 
\bigcup_{i=1}^m \overline{L_i}$ as required.
\end{proof}

\begin{remark}\label{rem:2}
On \cite[page~487]{is}, $\alpha:Y\subseteq H\rightarrow G$ is defined to be ``continuous piecewise affine'' when $\alpha$ is piecewise affine, and $Y$ is clopen in $H$.  If $\alpha$ is of this form, then we can extend $\alpha$ to a map $\alpha:H\rightarrow G_\infty$ by defining $\alpha(y) = \infty$ for $y\not\in Y$, and then $\alpha$ will still be continuous, because $Y$ is clopen.  Then we are in exactly the situation of Proposition~\ref{prop:2}, and so the results above imply that $\alpha$ is a continuous piecewise affine map in our sense.

As such, the use of \cite[Lemma~1.3(ii)]{is} in the proof of the converse of the result above, \cite[Proposition~3.1]{is}, is also corrected.
\end{remark}

The original use of \cite[Lemma~1.3(ii)]{is} was to show that if $\alpha:Y\subseteq G\rightarrow H$ is piecewise affine, and continuous, with $Y$ open, then also $\overline Y$ is open, and there is a continuous piecewise affine map $\overline\alpha : \overline Y\rightarrow H$ extending $\alpha$.  We have been unable to decide if this result is true or not.

\bigskip

\noindent\emph{Author's Address:}

\noindent
Matthew Daws, Jeremiah Horrocks Institute,
University of Central Lancashire,
Preston,
PR1 2HE,
United Kingdom

\noindent\emph{Email:} \texttt{matt.daws@cantab.net}

\end{document}